\documentclass[reqno,12pt]{amsart}
\usepackage{amssymb,graphicx,color,amsthm,amsmath,cite}%

\theoremstyle{plain}
\newtheorem{thm}{Theorem}[section]
\newtheorem{cor}[thm]{Corollary}
\newtheorem{lem}[thm]{Lemma}
\newtheorem{rem}[thm]{Remark}
\newtheorem{prop}[thm]{Proposition}
\numberwithin{equation}{section}

\newcommand{\myIm}{\mathop{\rm Im}}
\newcommand{\myRe}{\mathop{\rm Re}}
\newcommand{\dom}{\mathop{\rm dom}}
\newcommand{\norm}[1]{\left\Vert#1\right\Vert}

\newcommand\wt{\widetilde}

\newcommand\bR{{\mathbb R}}
\newcommand\bC{{\mathbb C}}

\newcommand\eps{\varepsilon}

\newcommand{\xim}{-\xi_\eps}
\newcommand{\xip}{\xi_\eps}
\newcommand{\xipm}{\pm\xi_\eps}
\newcommand{\ximp}{\mp\xi_\eps}

\newcommand{\xp}{x_\eps}

\newcommand\rmd{\mathrm{d}}
\newcommand\rme{\mathrm{e}}
\newcommand\rmi{\mathrm{i}}

\newcommand{\dint}{\displaystyle\int}

\begin{document}

\title[Singularly scaled Schr\"{o}dinger operators]{Norm resolvent convergence of singularly scaled Schr\"{o}dinger operators and  $\delta'$-potentials}

\author{Yu. D.~Golovaty \and R. O.~Hryniv}
\email{yu\_\,holovaty@franko.lviv.ua, rhryniv@iapmm.lviv.ua}

\address[YuG]{Department of Mechanics and Mathematics,
  Ivan Franko National University of Lviv, 1 Universytetska str., 79000 Lviv, Ukraine }

\address[RH]{Institute for Applied Problems of Mechanics and Mathematics, 3b Naukova str.,
79060 Lviv, Ukraine \and Institute of Mathematics, the University of Rzesz\'{o}w, 16\,A Rejtana al., 35-959 Rzesz\'{o}w, Poland}


\begin{abstract}
For a real-valued function~$V$ from the Faddeev--Mar\-chen\-ko class, 
we prove the norm resolvent convergence, as $\eps\to0$, of a family $S_\eps$ of one-dimensional Schr\"odinger
operators on the line of the form
\[
    S_\eps:= -\frac{\rmd^2}{\rmd x^2} + \frac1{\eps^2}\,V\Bigl(\frac{x}{\eps}\Bigr).
\]
Under certain conditions the functions $\eps^{-2}V(x/\eps)$ converge in the sense of  distributions as $\eps\to 0$ to~$\delta'(x)$, and then the limit $S_0$ of~$S_\eps$ might be considered as a ``physically motivated'' interpretation of the one-dimensional Schr\"odinger operator with potential $\delta'$.
\end{abstract}
\maketitle



\section{Introduction}\label{sec:intr}

The aim of this paper is to study convergence as $\eps\to0$ of the family~$S_\eps$ of Schr\"odinger operators on the line given by
\begin{equation}\label{eq:intr.Seps}
    S_\eps := -\frac{\rmd^2}{\rmd x^2} + \frac{1}{\eps^2}\,V\Bigl( \frac{x}{\eps}\Bigr),
\end{equation}
for the largest possible class of potentials~$V$, namely, for real-valued~$V$ in the Faddeev--Marchenko class $L_1\bigl(\bR;(1+|x|)\,\rmd x\bigr)$. The asymptotic behaviour of~$S_\eps$ and their many-dimensional analogues has been discussed in both the mathe\-matical and physical literature in connection with the small-energy scatte\-ring~\cite{AlbeverioGesztesyHoeghKrohn:1982,JensenKato:1979} and singular perturbations~\cite{AlbeverioHoeghKrohn:1981,Seba:1985,SebRMP} since 1980-ies. Recently, the Schr\"odinger operator family~$S_\eps$ has been enjoying a renewed interest motivated by the question on approximation of thin quantum waveguides by quantum graphs~\cite{AlbeverioCacciapuotiFinco:2007,CacciapuotiExner:2007, CacciapuotiFinco:2007,Cacciapuoti:2011,DellAntonioPanati:2004, Gianesello:2011}.

In three dimensions, the corresponding family of Hamiltonians
\[
    H_\eps:=-\Delta + \frac{\lambda(\eps)}{\eps^2}\,V\Bigl(\frac{x}{\eps}\Bigr)
\]
was first studied by Albeverio and H{\o}egh-Krohn~\cite{AlbeverioHoeghKrohn:1981}. Here $\lambda(\eps)$ is a smooth function with $\lambda(0)=1$ and $\lambda'(0)\ne0$ and $V$ is short-range and of Rollnik class~\cite[Ch.~X.2]{Reed-Simon-I}. It was shown that the family $H_\eps$ converges as $\eps\to0$ in the strong resolvent sense to~$H_0$ that is either the free Hamiltonian $-\Delta$ or its perturbation by a delta-function
(see the pioneering work by Berezin and Faddeev \cite{BerezinFadeev}) depending on whether or not there is a zero-energy resonance for $H=-\Delta+V$. We recall that $H$ is said to possess a zero-energy resonance if the equation $H \psi =0$ has a distributional solution that is bounded but does not belong to $L_2(\bR^3)$. Analogous results were established in~\cite{AlbeverioHoeghKrohn:1981} also for $V$ containing finitely or infinitely many  summands scaled about different centres.

In~\cite{AlbeverioGesztesyHoeghKrohn:1982}, the authors discussed the low-energy scattering in two particle non-relativistic quantum mechanics. They used the results of~\cite{AlbeverioHoeghKrohn:1981} and the connection between the low-energy behaviour of the scattering amplitude and scattering matrix for the corresponding Hamiltonian $H=-\Delta + V$ in~$L_2(\bR^3)$ and for the scaled Hamiltonians
\(
    H_\eps=-\Delta + \eps^{-2} V(x/\eps)
\)
as $\eps \to 0$ to study in detail possible resonant and non-resonant cases. Similar problem for Hamiltonians including the Coulomb-type interaction was treated in~\cite{AlbeverioGesztesyHoeghKrohnStreit:1983}.

Interestingly enough, the low-energy scattering theory for Schr\"odinger operators in dimensions one and two is more complicated than in dimension three. This is connected with, respectively, the square root and logarithmic singularities
the Green function of the free Hamiltonian then possesses. The low-energy scattering for the one-dimensional Schr\"odinger operator~$S_1$ and its connection to the behaviour of~$S_\eps$ as $\eps\to0$ was thoroughly investigated by Boll\'e, Gesztesy, Klaus, and Wilk, both for the non-resonant~\cite{BolleGesztesyWilk:1985} and resonant \cite{BolleGesztesyKlaus:1987} cases respectively; in dimension two, the low-energy asymptotics was discussed in~\cite{BolleGesztesyDanneels:1988}. Continuity of the scattering matrix at zero energy for one-dimensional Schr\"odinger operators with Faddeev--Marchenko potentials in the re\-so\-nant case was independently established by Guseinov in~\cite{Guseinov:1985} and by Klaus in~\cite{Klaus:1988}.

Another reason to study the family~$S_\eps$ comes from the quantum graph theory.
One of the fundamental questions of this theory consists of justifying the possibility of approximating dynamics of a quantum particle confined to real-world mesoscopic waveguides of small width~$\eps$ by its dynamics on the idealized one-dimensional ``mani\-folds'' obtained in the limit as $\eps$ vanishes.

For instance, for bent waveguides $\Omega_\eps$ in dimension~$2$ that coincide with two straight strips outside a compact ``vertex'' region this question was studied in~\cite{AlbeverioCacciapuotiFinco:2007}. The authors showed that the problem reduces to  establishing the norm resolvent convergence of the operator family~$S_\eps$ of~\eqref{eq:intr.Seps} as $\eps\to0$. Assuming that $V$ has nonzero mean and decays exponentially, i.e., that
\begin{equation}\label{eq:intr.Vzeromean}
    \int_{\bR} V(x)\,\rmd x \ne 0 \quad \mbox{and}\quad
        e^{a|\,\cdot\,|}V \in L_1(\bR) \quad \text{for some}\quad a>0,
\end{equation}
norm resolvent convergence of~$S_\eps$ as $\eps\to0$ was proved and the limiting operator $S_0$ was identified. Depending on whether or not the potential $V$ is resonant, the operator~$S_0$ is either the direct sum $S_-\oplus S_+$ of two free half-line Schr\"odinger operators given by $S_\pm:=-\tfrac{d^2}{dx^2}$ on $\bR_\pm$ and subject to the Dirichlet condition $y(0)=0$ at the origin, or a singular perturbation $S(\theta)$ of the free Schr\"odinger operator~$S$ defined by
\begin{equation}\label{eq:intr.Stheta}
    S(\theta)y=-y''
\end{equation}
on functions $y$ in~$W_2^2(\bR\setminus\{0\})$ obeying the nontrivial interface conditions at the origin,
\begin{equation}\label{eq:intr.Stheta-dom}
    y(0+) = \theta y(0-),
        \quad \theta y'(0+) = y'(0-).
\end{equation}
The number~$\theta$ depends on the geometric properties (in particular, on the curvature) of the bent waveguide. The convergence results for bent waveguides were recently re-examined from a different viewpoint in the paper~\cite{Cacciapuoti:2011}.

In~\cite{CacciapuotiExner:2007}, the analysis of~\cite{AlbeverioCacciapuotiFinco:2007} was further extended to the case of waveguides~$\Omega_\eps$ with non-trivial scaling properties around the vertex region; namely, the authors demonstrated that in the resonant case the class of limiting Hamiltonians is wider and includes, in particular, perturbations of the free Schr\"odinger operator~$S$ by the delta-functions. Results analogous to those of~\cite{AlbeverioCacciapuotiFinco:2007} and \cite{CacciapuotiExner:2007} were also established in~\cite{CacciapuotiFinco:2007} in the case of Robin conditions on the boundary of~$\Omega_\eps$.

The effect of twisting in 3D quantum wave-guides with shrinking non-circular cross-section was investigated in \cite{Gianesello:2011}. The problem was reduced to the study of convergence of singularly scaled Schr\"{o}dinger-type operators $\mathcal{H}_\eps$ in an unbounded tube with potentials containing the singular term $\eps^{-2} \vartheta(x_1/\eps)$, where $x_1$ is a longitudinal coordinate in the tube.
The function $\vartheta$ of compact support describes the ``fast'' twisting of the cross-section over a small $x_1$-interval shrinking to a point $x_0$. It was proved that the family $\mathcal{H}_\eps$ converges as $\eps\to 0$ in the (suitably understood) norm resolvent sense to the operator $\mathcal{H}_0$, which is a one-dimensional Schr\"{o}dinger operator subject to
the Dirichlet boundary condition at $x_0$.
Earlier, the operators $\mathcal{H}_\eps$ appeared in \cite{KrejcirikZuazua:2010}
in connection with the problem of large-time behaviour of solutions to the heat equation in a twisted tube, and their strong resolvent convergence was established therein.

The limiting behaviour of the family $S_\eps$ is interesting for yet another reason. Namely, in the case where $V$ has a zero mean and its first moment is~$-1$, i.e., where
\[
    \int_{\bR} V(x)\,\rmd x =0, \qquad \int_{\bR} xV(x)\,\rmd x =-1,
\]
the family of scaled potentials $V_\eps(x):= \eps^{-2}V(\eps^{-1}x)$ converges as $\eps\to0$ to the derivative $\delta'$ of the Dirac delta-function. Therefore, the limit~$S_0$ of $S_\eps$ can be regarded as a physically motivated realization of the free Hamiltonian~$S$ perturbed by the \emph{singular potential} $\delta'$.
We notice that there is a related but completely different notion of $\delta'$-\emph{interaction}; namely, following a widely accepted agreement, the Hamiltonian $S_{\beta,\delta'}$ with $\delta'$-\emph{interaction} that is given formally by
\[
    -\frac{\rmd^2}{\rmd x^2} - \beta\langle\,\cdot\,,\delta'\rangle \,\delta',
\]
should be interpreted as the free Schr\"odinger operator $S_{\beta,\delta'}y=-y''$ acting on the domain
\begin{eqnarray}
   &\dom S_{\beta,\delta'}:=\bigl\{y \in W_2^2(\bR\setminus\{0\})\mid y'(0-)=y'(0+)=:y'(0),&\nonumber\\
    &\phantom{mmmmmmmmmmmmmmmmm} y(0+)-y(0-)=\beta y'(0)\bigr\},&\nonumber
\end{eqnarray}
see~\cite{Albeverio2edition,AlbeverioKurasov}. However, there is no clear physical motivation for this particular choice and, moreover, $S_{\beta,\delta'}$ cannot be used for defining a free Hamil\-to\-nian with $\delta'$-\emph{potential} \cite{NizhFAA2006}. We note that $\delta'$-interactions and, more generally, singular point interactions for Schr\"odinger operators in dimension one and higher have widely been discussed in both the mathematical and physical literature; see \cite{Antonevich:1999,BrascheFigariTeta,ExnerNeidhardtZagrebnov,IsmagilovKostyuchenko,
KostenkoMalamud:2009,KronigPenney,Mikhailets,ShkalikovSavchukTMMO2003} and also the extensive bibliography lists in the monographs \cite{Albeverio2edition,AlbeverioKurasov,Koshmanenko:1999}.

\v{S}eba~\cite{Seba:1985} was seemingly the first to realize the connection between the limiting behaviour of $S_\eps$ and $\delta'$-perturbations of the Schr\"odinger operators. The paper~\cite{Seba:1985} discussed an even wider class of problems including $S_\eps$ as a particular case; however, its results for the family~\eqref{eq:intr.Seps} erroneously state that the only possible norm-resolvent limit of~$S_\eps$ is the direct sum $S_-\oplus S_+$ of the free half-line Schr\"odinger operators subject to the Dirichlet boundary condition at $x=0$. Such a result would suggest that in  dimension~$1$ no non-trivial definition of the Schr\"odinger operator with potential~$\delta'$ is possible. Recalling that the Schr\"odinger operators are quantum mechanical Hamiltonians for a particle on the line, one would have to conclude that, in dimension~$1$, the $\delta'$ potential barrier is completely opaque, i.e., that the particle cannot tunnel through it.

However, such a conclusion is in contradiction with the numerical analysis of exactly solvable models of~\eqref{eq:intr.Seps} with piece-wise constant~$V$ of compact support performed recently by Zolotaryuk a.o.~\cite{ChristianZolotarIermak03}, \cite{ZolotarChristianIermak06,Zolotaryuk08,Zolotaryuk09,Zolotaryuk10}.
Namely, the authors demonstrated that for resonant~$V$, the limiting value of the transmission coefficient~$T_\eps(k)$ of the operator~$S_\eps$ is different from zero, thus indicating that the limiting operator~$S_0$ cannot be given by $S_-\oplus S_+$. In certain cases, $S_0$ was identified with the operator~$S(\theta)$ of~\eqref{eq:intr.Stheta}, \eqref{eq:intr.Stheta-dom}.
The operator of the form $S(\theta)$ also appear  in \cite{KurasovScrinziElander:1994, NizhFAA2006} as a realization of the pseudo-Hamiltonian $-\frac{\rmd^2}{\rmd x^2} + \alpha\delta'(x)$ by means of
the distribution theory  over discontinuous test functions.
Yet another evidence that the convergence result of~\cite{Seba:1985} cannot be true was derived in  the paper~\cite{GolovatyManko:2009}, where eigenvalue and eigenfunction asymptotics as $\eps\to0$ were studied for the full-line Schr\"odinger operators given by the differential expression
\[
    -\frac{\rmd^2}{\rmd x^2} + \frac1{\eps^2}V\Bigl(\frac{x}\eps\Bigr) + W(x),
\]
where $V$ is regular and of compact support and $W$ is unbounded at infinity. Namely, the eigenfunctions were shown to satisfy in the limit the Dirichlet condition~$y(0)=0$ in the non-resonant case and the interface condition~\eqref{eq:intr.Stheta-dom} in a resonant case, thus again exhibiting the zero-energy dichotomy.

The above results motivated us to re-examine the convergence of the Schr\"odinger operator family~\eqref{eq:intr.Seps} for real-valued~$V$ not necessarily satisfying~\eqref{eq:intr.Vzeromean}. In our previous paper~\cite{GolovatyHryniv:2010}, we treated the case where $V$ is of compact support, taken~$[-1,1]$ for definiteness, and, in particular, confirmed the result established in~\cite{AlbeverioCacciapuotiFinco:2007,CacciapuotiExner:2007,CacciapuotiFinco:2007} under the restriction~\eqref{eq:intr.Vzeromean}. In~\cite{GolovatyHryniv:2010, GolovatyManko:2009}, a real-valued potential~$V$ was called \emph{resonant} if the Neumann Sturm--Liuoville operator~$\mathcal{N}$ given by
\[
    \mathcal{N} y := - y'' + V y
\]
on functions in the Sobolev space~$W_2^2(-1,1)$ satisfying the conditions $y'(-1)=0$, $y'(1)=0$ has a non-trivial null-space. This definition of a resonant case agrees with the one taken in~\cite{AlbeverioCacciapuotiFinco:2007}. We proved in~\cite{GolovatyHryniv:2010} that $S_\eps$ converge as $\eps\to0$ in the norm resolvent sense to the limit operator~$S_0$ that is equal to $S_-\oplus S_+$ in the non-resonance case and to $S(\theta)$ of \eqref{eq:intr.Stheta}, \eqref{eq:intr.Stheta-dom} in the resonance case with $\theta:=u(1)/u(-1)$, where $u$ is an eigenfunction of $\mathcal{N}$ corresponding to the eigenvalue zero. The eigenfunction is  unique up to a constant factor, so that the number~$\theta$ is well defined.

The main goal of this paper is to extend the convergence result of~\cite{GolovatyHryniv:2010} to the set of real-valued potentials~$V$ of the Faddeev--Marchenko class. In particular, we do not assume that $V$ is of non-zero mean or that $V$ decays exponentially. Nevertheless, we shall prove that the results on convergence of~$S_\eps$ established in~\cite{AlbeverioCacciapuotiFinco:2007,CacciapuotiExner:2007,CacciapuotiFinco:2007,GolovatyHryniv:2010} for particular cases remain valid.

Recall~\cite{Klaus:1982} that the Schr\"odinger operator~$S_1=-\frac{\rmd^2}{\rmd x^2}+V$ is said to possess a \emph{zero-energy resonance} (or \emph{half-bound state}) if there exists a solution~$y$ to the equation
\[
    - y'' +Vy= 0
\]
that is bounded on the whole line; the corresponding potential~$V$ is then called \emph{resonant}. Such a solution~$y$ is then unique up to a scalar factor and has nonzero limits $y(\pm\infty)$ at $\pm\infty$, so that the number~
\begin{equation}\label{eq:intr.theta}
    \theta:=\frac{y(+\infty)}{y(-\infty)}
\end{equation}
is well defined. Our main result reads as follows.

\begin{thm}\label{thm:main}
Assume that $V$ is real valued and belongs to the Faddeev--Marchenko class. Then the operator family~\eqref{eq:intr.Seps} converges as $\eps\to0$ in the norm resolvent sense, and the limit $S_0$ is equal to the direct sum $S_-\oplus S_+$ of the Dirichlet half-line Schr\"odinger operators~$S_\pm$ in the non-resonant case, and to the operator~$S(\theta)$ defined by~\eqref{eq:intr.Stheta}, \eqref{eq:intr.Stheta-dom}, with~$\theta$ of~\eqref{eq:intr.theta}, in the resonant case.
\end{thm}

We establish Theorem~\ref{thm:main} in two steps. On the first step, we adapt our approach of~\cite{GolovatyHryniv:2010} and establish the theorem with the potential~$V_\eps:=\eps^{-2}V(\eps^{-1}\,\cdot\,)$ truncated onto the contracting intervals~$(-x_\eps,x_\eps)$, for a suitably defined $x_\eps$. On the second step, we show that the part of the potential~$V_\eps$ outside $(-x_\eps,x_\eps)$ introduces a perturbation that is too weak to affect the limit of the resolvents as $\eps\to0$. In both the resonant and non-resonant cases, the reasoning is based on a careful asymptotic analysis of the Jost solutions for $S_\eps$ and $S_1$ at infinity and in the vicinity of the splitting points~$\pm x_\eps$.

It should be noted that the above theorem seemingly cannot be extended beyond the Faddeev--Marchenko class of potentials~$V$. Indeed, take $V(x)= (1+x^2)^{-1}$; then the family $V_\eps(x):=(\eps^2 + x^2)^{-1}$ increases as $\eps$ decreases to zero. The quadratic forms $\mathfrak{s}_\eps$ generated by the Schr\"odinger operators $S_\eps$ are closed and well defined on~$W_2^1(\bR)$. For $y \in W_2^1(\bR)$ the limit
\begin{eqnarray}
    \lim_{\eps\to0+} \mathfrak{s}_\eps(y,y) &=& \displaystyle\int_{\bR}|y'(x)|^2\,\rmd x
        + \lim_{\eps\to0+} \int_{\bR}\frac{|y(x)|^2\,\rmd x}{\eps^2 +x^2}\nonumber\\
            &=& \displaystyle \int_{\bR}|y'(x)|^2\,\rmd x
        + \int_{\bR} \frac{|y(x)|^2\,\rmd x}{x^2} =: \mathfrak{s}_0(y,y)\nonumber
\end{eqnarray}
is finite if and only if $y(0)=0$; the ``only if'' part is clear, and the ``if'' part follows from the fact that the Hardy operator,
\[
        g \mapsto \frac1x\int_0^x g(t)\,\rmd t
\]
is bounded in~$L_2(0,1)$~\cite[Ch.~9.9]{HLP}. By Theorem~S.14 of~\cite{Reed-Simon-I}, the quadratic form~$\mathfrak{s}_0$ is closed on the domain~$\{y\in W_2^1(\bR) \mid y(0)=0\}$ and  the Schr\"odinger operators~$S_\eps$ converge in the strong resolvent sense as $\eps\to0+$ to the operator~$S_0$ corresponding to the quadratic form~$\mathfrak{s}_0$. Since $\dom S_0 \subset \dom \mathfrak{s}_0$ by the first representation theorem~\cite[Theorem~VI.2.1]{Kato}, we conclude that $S_0$ is the direct sum of two half-line Schr\"odinger operators $S_-$ and $S_+$ with the Bessel potential $1/x^2$. Thus there is no resonance effect present here and the limiting operator contains a nontrivial potential, in contrast to the situation of Theorem~\ref{thm:main}.

It is worth noting that the operators $S_\eps$ given by \eqref{eq:intr.Seps} are scale-invariant and therefore it should be expected that the limit operator $S_0$ possesses the same property. And indeed, both in the non-resonant and resonant cases the limit operators $S_-\oplus S_+$ and $S(\theta)$ describe the so-called scale-invariant point interactions~\cite{FulopTsutsuiCheon:2003}.

We also remark that one can construct infinitely many resonant potentials starting with an arbitrary compactly supported $V\in L_1(\bR)$ and considering a family $\{\alpha V\}$ with a real coupling constant~$\alpha$. As shown in~\cite{GolovatyHryniv:2010} (see also~\cite{ChristianZolotarIermak03}, \cite{ToyamaNogami}--\cite{Zolotaryuk10} for various step-like potentials $V$), there is an infinite discrete set $\Sigma$ of $\alpha$'s such that $\alpha V$ is resonant for every $\alpha\in\Sigma$. The same conclusion holds if $V$ is supported by the whole line but decays rapidly at infinity. We are going to discuss this question elsewhere.

The rest of the paper is organized as follows.
In Section~\ref{sec:pre}, we recall the definition and some important properties of the Jost solutions of the Schr\"odinger equation; also, the low-energy and large-$x$ asymptotics of the Jost solutions and of some special solutions are established therein.
In Section~\ref{sec:aux-fam}, we study the norm resolvent convergence of the auxiliary family of  Schr\"odinger operators~$\wt S_\eps$ with potentials $\chi_\eps V_\eps$, where $\chi_\eps$ is the characteristic function of a neighbourhood of the origin squeezing to zero as $\eps\to0$.  Section~\ref{sec:Jost} is devoted to investigation of the resolvent of $\wt S_\eps$ via the subtle asymptotic analysis of the Jost solutions of this ope\-rator. Simultaneously,  we describe the limit behaviour  of scattering coefficients for the ope\-rator~$\wt S_\eps$. In these three sections,  the non-resonant and resonant cases have to be treated separately.
Finally, in the last section we  establish proximity of the operator families $S_\eps$ and $\wt S_\eps$ in the norm resolvent sense, which, in combination with Theorem~\ref{thm:conv-tildeS},
gives a complete proof of the main result.

\smallskip\noindent\emph{Notations.} Throughout the paper,  $C_1$, $C_2$, $C_3$, $C_4$, and $C_5$ shall stand for the constants of Proposition~\ref{pro:fpm} and Lemma~\ref{lem:f+-asymptoticsResonance}. Letters $c_j$ denote various posi\-ti\-ve numbers independent of~$\eps$, whose values might be different in different proofs, and $\|f\|$ stands for the $L_2(\bR)$-norm of a function~$f$.

\section{Asymptotics  of Jost solutions}\label{sec:pre}

\subsection{Large-$x$ behaviour at low energies}
Throughout the paper $V$ denotes a fixed real-valued potential of the Faddeev--Marchenko class, i.e., satisfying
\begin{equation*}\label{eq:Jost.FM}
    \int_\mathbb{R} (1+ |t|)|V(t)|\,\rmd t < \infty.
\end{equation*}
Next, we denote by $f_+(\,\cdot\,,k)$ and $f_-(\,\cdot\,,k)$ respectively the \emph{right} and \emph{left Jost solutions} of the Schr\"odinger equation
\begin{equation}\label{eq:Jost.de}
    -y'' + V y = k^2 y
\end{equation}
for complex $k$ with $\myIm k \ge0$. The Jost solution $f_+(x,k)$ is asymptotic to $\rme^{\rmi kx}$ as $x \to +\infty$, while $f_-(x,k)$ is asymptotic to~$\rme^{-\rmi kx}$ as $x\to -\infty$. For non-zero~$k$, the Jost solutions exist whenever the potential $V$ is integrable at $\pm\infty$~\cite[Ch.~I.1.3]{ChadanSabatier:1989} or even under somewhat weaker assumptions~\cite{Ker:2007}; however, for $V$ of the Faddeev--Marchenko class they have some special properties.
Namely, set
\begin{eqnarray}
    &\sigma_-(x):=\dint_{-\infty}^x |V(t)|\,\rmd t,\qquad
     &\tau_-(x) := \int_{-\infty}^x (1+|t|)|V(t)|\,\rmd t,\nonumber\\
    &\sigma_+(x):=\dint_x^{\infty} |V(t)|\,\rmd t, \qquad
     &\tau_+(x):=\int_x^{\infty} (1+|t|)|V(t)|\,\rmd t;\nonumber
\end{eqnarray}
then the claims of \cite[Lemma~1]{DT} and \cite[Lemma~3.1.3]{Marchenko:1986} can be stated as follows.

\begin{prop}\label{pro:fpm}
Assume that the potential $V$ is of the Faddeev--Marchenko class.  Then there are numbers $C_j$,
$j=1,\dots,4$, such that the following holds for every $k$ in the closed upper-half complex plane:
\begin{eqnarray}
 & |f_+(x,k) - \rme^{\rmi kx}|    \le  C_1 |\rme^{\rmi kx}|\, \tau_+(x),\qquad\qquad
                          &x\ge0; \label{eq:f+0} \\
               &|f_+(x,k) - \rme^{\rmi kx}| \le  C_2 |\rme^{\rmi kx}|\bigl(1 +|x|\bigr),\quad\qquad
                          &x<0, \label{eq:f+R} \\
   &\;\;\:|f'_+(x,k) - \rmi k f_+(x,k)|\le  C_3 |\rme^{\rmi kx}|\,\sigma_+(x),\qquad
                        &x\ge0  \label{eq:f+'0},\\
   &\,|f'_+(x,k) - \rmi k f_+(x,k)| \le  C_4 |\rme^{\rmi kx}|,\qquad\qquad
                         &x<0\label{eq:f+'R}.
\end{eqnarray}
Similar estimates hold for $f_-(\,\cdot\,,k)$.
\end{prop}

In what follows, we shall consider restrictions of $V_\eps$ onto contracting intervals $(x_\eps,x_\eps)$, the choice of $x_\eps$ being tailored for the given~$V$ in a special way. We start with the following observation.

\begin{lem}\label{lem:FMwithRho}
For $V$ of the Faddeev--Marchenko class
  there exists an even, conti\-nuous, and positive function $\rho_V\colon \bR\to \bR$ such that
  \begin{itemize}
    \item[(i)] $\rho_V$ strictly increases for $x>0$;
    \item[(ii)]  $|x|^{-1}\rho_V(x)\to +\infty$ as $|x|\to +\infty$;
    \item[(iii)] $\int_\bR \rho_V(x)|V(x)|\,\rmd x < \infty$.
  \end{itemize}
 \end{lem}
\begin{proof} For $V$ of compact support, set $\rho_V(x):=1+x^2$. Otherwise the function
\[
    \tau(x):=\tau_+(|x|) + \tau_-(-|x|)=\int_{|t|>|x|} \bigl(1+|t|\bigr)|V(t)|\,\rmd t
\]
is strictly positive and even; moreover, it does not increase for $x>0$ and vanishes at infinity. We set
\[
    \rho_V(x):= \frac{1+|x|}{\tau^\alpha(x)}
\]
for $\alpha\in(0,1)$; then \textit{(i)} and \textit{(ii)} are immediate, while \textit{(iii)} follows from
\begin{eqnarray}\nonumber
    \int_{\bR} \rho_V(x)|V(x)|\,\rmd x
        &=& \int_0^{\infty} \frac{(1+x)\left(|V(x)|+ |V(-x)|\right)}{\tau^{\alpha}(x)}\,\rmd x\\
    &=&-\int_0^{\infty} \frac{\tau'(x)}{\tau^{\alpha}(x)}\,\rmd x
    =\frac{\tau^{1-\alpha}(0)}{1-\alpha}<\infty.\nonumber
\end{eqnarray}
The proof is complete.
\end{proof}

Fix $\rho_V$ as in  Lemma~\ref{lem:FMwithRho} and for $\eps>0$ denote by $\xi_\eps$ the unique positive solution of the equation $\rho_V(\xi) = 1/\eps$.  Such a solution $\xi_\eps$  exists for all  sufficiently small positive $\eps$. It follows from (i) and (ii) of  Lemma~\ref{lem:FMwithRho} that $\xi_\eps\to+\infty$ and $x_\eps:=\eps\xi_\eps\to0$ as $\eps\to0$. This choice of $x_\eps$ and $\xi_\eps$ is fixed for the rest of the paper.

Set $W\{f,g\}:=fg'-f'g$ to be the Wronskian of functions $f$ and $g$ and define $D(k)$ as the Wronskian of the Jost solutions $f_+(\,\cdot\,,k)$ and $f_-(\,\cdot\,,k)$,
\[
    D(k) := W\{f_+(\,\cdot\,,k),f_-(\,\cdot\,,k)\} = f_+(x,k)f'_-(x,k) - f'_+(x,k)f_-(x,k).
\]
Fix a nonzero $k \in \overline{\bC^+}$; then Proposition~\ref{pro:fpm} implies the following asymp\-to\-tic behaviour of the Jost solutions.

\begin{lem}\label{lem:f+-asymptotics}
The following holds as $\eps \to 0$:
\begin{eqnarray}         \label{eq:f+-xi_eps}
    &f_+(\xip,\eps k) \to 1,  \qquad\quad \qquad         &f_-(\xim,\eps k) \to 1;\\
    &\eps f_+(\xim,\eps k) \to 0, \quad\;\;\qquad      &\eps f_-(\xip,\eps k) \to 0;    \label{eq:f-+xi_eps}\\
    &\eps^{-1}f'_+(\xip,\eps k) \to \rmi k,\;\;\;\;\qquad   &\eps^{-1}f'_-(\xim,\eps k) \to -\rmi k;   \label{eq:f'+-xi_eps}\\
    &f'_+(\xim,\eps k)\to -D(0), \qquad        &f'_-(\xip,\eps k)\to D(0).     \label{f'-+xi_epsD0}
\end{eqnarray}
\end{lem}

\begin{proof}
  Relations \eqref{eq:f+-xi_eps} follow immediately from \eqref{eq:f+0} and an analogous estimate for $f_-$. Also,  \eqref{eq:f+R} and a similar inequality for $f_-$ imply that
  $$
  |f_+(\xim,\eps k)|+|f_-(\xip,\eps k)|\le  c (1+\xip)
  $$
  for some $c>0$ independent of~$\eps$;
  since $\eps\xip\to 0$ as $\eps\to 0$, formulae \eqref{eq:f-+xi_eps} follow.

  Next, estimate \eqref{eq:f+'0} yields
  \begin{equation}\label{est:f+-ik}
    |\eps^{-1}f'_+(\xip,\eps k) - \rmi k |\le  C_3 \eps^{-1}\sigma_+(\xip)+|k|\,|f_+(\xip,\eps k)-1|.
  \end{equation}
 Now the choice of the~$\xip$ and the properties of the function~$\rho_V$ justified in Lemma~\ref{lem:FMwithRho} show that
 \begin{eqnarray}
    \eps^{-1}\sigma_+(\xip)&=&\eps^{-1}\int_{\xip}^{\infty} |V(t)|\,\rmd t\nonumber\\
    &&\leq \eps^{-1}\int_{\xip}^{\infty}\frac{\rho_V(t) |V(t)|}{\rho_V(\xip)}\,\rmd t= \int_{\xip}^{\infty}\rho_V(t) |V(t)|\,\rmd t\to 0\nonumber
 \end{eqnarray}
as $\eps\to 0$. A passage to the limit in \eqref{est:f+-ik} establishes the first relation of~\eqref{eq:f'+-xi_eps}; the second one is justified similarly.

Finally, combination of the equalities
\[
   D(\eps k)=
   \left|\begin{array}{lr}
     \eps f_+(\xim,\eps k) &  \phantom{\eps^{-1}}f_-(\xim,\eps k)\\
     \phantom{\eps} f_+'(\xim,\eps k) & \eps^{-1} f_-'(\xim,\eps k)
   \end{array}\right|=
   \left|\begin{array}{lr}
      \phantom{\eps^{-1}}f_+(\xip,\eps k) & \eps f_-(\xip,\eps k)\\
      \eps^{-1} f_+'(\xip,\eps k) & \phantom{\eps}f_-'(\xip,\eps k)
   \end{array}\right|
\]
with relations \eqref{eq:f+-xi_eps}--\eqref{eq:f'+-xi_eps} proved above results in~\eqref{f'-+xi_epsD0}. The proof is complete.
\end{proof}

\subsection{Refinement in the resonant case}

In the resonant case, the Jost solutions $f_+$ and $f_-$ become linearly dependent at $k=0$,
i.e.,
\begin{equation}\label{eq:theta}
    f_-(\,\cdot\,,0) = \theta f_+(\,\cdot\,,0)
\end{equation}
for some real nonzero~$\theta$. The corresponding half-bound state $y$ can be given as $y=c_- f_-(\,\cdot\,,0)$ or $y=c_+ f_+(\,\cdot\,,0)$, with some constants $c_\pm$ satisfying the relation $c_+=\theta c_-$. Since  $y(\pm\infty)=c_\pm$ by the definition of the Jost solutions, the number~$\theta$ as defined by~\eqref{eq:theta} is the same as in \eqref{eq:intr.theta}.

Moreover, in this case, more precise information on the asymptotic behaviour of $f_\pm(\,\cdot\,,0)$ at infinity is available.  Namely, we shall prove that the solutions $f_\pm(\,\cdot\,,\eps k)$ remain then bounded over $(-\xi_\eps,\xi_\eps)$ uniformly in $\eps$ and shall re-examine the convergence in \eqref{eq:f-+xi_eps} and \eqref{f'-+xi_epsD0}.
To do this, we need another pair of solutions to equation~\eqref{eq:Jost.de}.

\begin{lem}\label{lem:g+-asymptotics}
 There exist two solutions $g_+(\,\cdot\,, k)$ and $g_-(\,\cdot\,, k)$ of \eqref{eq:Jost.de} satisfying the relations $W\{f_+,g_+\}=1$ and $W\{f_-,g_-\}=-1$ and such that, for some $x_+>0$, $x_-<0$ and all $k\in\overline{\bC^+}$, the following inequalities hold:
\begin{eqnarray}
    |g_+(x,k)| &\le  6 |\rme^{-\rmi kx}|\, x,
                         \qquad &x>x_+, \label{eq:g+xi_eps}\\
     |g_-(x,k)| &\le  6 |\rme^{\rmi kx}|\,|x|,
                         \qquad &x<x_-. \label{eq:g-xi_eps}
\end{eqnarray}
\end{lem}

\begin{proof}
It follows from the estimate~\eqref{eq:f+0} that there is an $x_+>0$ such that $f_+(x,k)\ne0$ for every $k\in\overline{\bC^+}$ and every $x\ge x_+$. Then the function
\[
    g_+(x,k):= f_+(x,k) \int_{x_+}^x \frac{\rmd t}{f_+^2(t,k)}
\]
is well defined for $x\ge x_+$, solves there equation~\eqref{eq:Jost.de}, and verifies the relation $W\{f_+,g_+\}=1$. Clearly, $g_+$ can uniquely be continued to the whole axis as a solution of~\eqref{eq:Jost.de}.

Without loss of generality we can (and shall) assume that $x_+$ is so large that $C_1 \tau_+(x) \le \tfrac12$  for all $k\in\overline{\bC^+}$ and all $x> x_+$, $C_1$ being the constant of~\eqref{eq:f+0}. Then
$$
 \frac12|\rme^{\rmi kx}| \le |f_+(x,k)| \le
 \frac32|\rme^{\rmi kx}|
$$
for such $k$ and $x$; therefore,
\[
    \frac{|\rme^{\rmi k x}g_+(x,k)|}{x}  \le \frac{6 |\rme^{2\rmi kx}|}{x-x_+} \int_{x_+}^x
        \frac{\rmd t}{|\rme^{2\rmi kt}|} \le 6,
\]
thus yielding~\eqref{eq:g+xi_eps}.

By similar arguments, there exists an $x_-<0$ and a solution $g_-(\,\cdot\,,k)$ of~\eqref{eq:Jost.de} that for all $x<x_-$ is equal to
\[
    g_-(x,k):= f_-(x,k) \int_{x}^{x_-} \frac{\rmd t}{f_-^2(t,k)}
\]
and satisfies the relation $W\{f_-,g_-\}=-1$ and the inequality~\eqref{eq:g-xi_eps}.
\end{proof}

\begin{lem}\label{lem:f+-asymptoticsResonance}
 Assume that $V$ is resonant and define $\theta$ via~\eqref{eq:theta}.
 Then for some $C_5>0$ and all $\eps$ small enough we have
\begin{equation}\label{eq:estF-+Resonance}
    \max_{|x|\leq \xi_\eps} |f_\pm(x,\eps k)| \leq C_5.
\end{equation}
Furthermore, the following holds as $\eps \to 0$:
\begin{eqnarray}
        \label{eq:f+-res}
    &f_+(-\xi_\eps,\eps k) \to \theta^{-1}, \qquad  &f_-(\xi_\eps,\eps k) \to \theta,\\
    &\eps^{-1} f'_+(-\xi_\eps,\eps k) \to \rmi k \theta,  \qquad
    &\eps^{-1} f'_-(\xi_\eps,\eps k) \to -\rmi k \theta^{-1}.\label{eq:f'+-res}
\end{eqnarray}
\end{lem}
\begin{proof} We shall only prove~\eqref{eq:estF-+Resonance} for the solution $f_+$, since the proof for~$f_-$ is completely analogous. Clearly, there are $\alpha_\eps$ and $\beta_\eps$ such that
\[
    f_+(\,\cdot\,,\eps k)
        = \alpha_\eps f_-(\,\cdot\,,\eps k) + \beta_\eps g_-(\,\cdot\,,\eps k),
\]
and it is easy to see that $\beta_\eps = D(\eps k)$.
Recall that $D(0)=0$ in the resonant case; moreover, equation~(2.29) of~\cite{Klaus:1988} implies that $D$ is differentiable at $k=0$ and that the derivative $\dot{D}(0)$ is
\[
\dot{D}(0)=-\rmi(\theta+\theta^{-1})\neq 0.
\]
It follows that
\begin{equation}\label{eq:limD(ek)/e}
   \lim_{\eps\to 0} \frac{D(\eps k)}{\eps}=-\rmi k(\theta+\theta^{-1});
\end{equation}
thus $\beta_\eps =O(\eps)$ as $\eps\to0$, and in combination with~\eqref{eq:g-xi_eps} this shows that
\begin{equation}\label{eq:f+alphaf-}
    \max_{x\in[-\xi_\eps,x_-]}|f_+(x,\eps k)-\alpha_\eps f_-(x,\eps k)|
            \leq 6|\beta_\eps| |\rme^{\rmi x_\eps k}| \xip = O(x_\eps)
\end{equation}
as $\eps\to0$.
For each fixed $x$, $f_\pm(x,k)$ are continuous functions of $k\in\overline{\bC^+}$.
We conclude from this and the relation~\eqref{eq:f+alphaf-} that $\alpha_\eps\to \theta^{-1}$ as $\eps\to0$. Therefore there exists $c>0$ such that, for all $\eps$ small enough,
\[
    \max_{x\in[-\xi_\eps,x_-]}|f_+(x,\eps k)|
        \leq 2\theta^{-1} \max_{x\in[-\xi_\eps,x_-]}|f_-(x,\eps k)|+1
            \leq c.
\]
A similar bound over $x_-\leq x \leq \xi_+$ follows from~\eqref{eq:f+0} and \eqref{eq:f+R}, thus establishing~\eqref{eq:estF-+Resonance} for $f_+$.

Specifying now~\eqref{eq:f+alphaf-} to $x=-\xi_\eps$ and recalling \eqref{eq:f+-xi_eps}, we get
\[
    \lim_{\eps\to0} f_+(-\xi_\eps,\eps k) = \theta^{-1}.
\]
Observing that
\begin{eqnarray}\nonumber
   &\displaystyle f_+(-\xi_\eps,\eps k)\,\frac{f'_-(-\xi_\eps,\eps k)}{\eps}
        - \frac{f'_+(-\xi_\eps,\eps k)}{\eps}\,f_-(-\xi_\eps,\eps k)\\\nonumber
       &\lefteqn{= \frac{D(\eps k)}\eps \to -\rmi k(\theta + \theta^{-1})}
\end{eqnarray}
as $\eps\to0$ by~\eqref{eq:limD(ek)/e}, we deduce from  Lemma~\ref{lem:f+-asymptotics} and \eqref{eq:f+-res} that
\[
    \lim_{\eps\to0}\eps^{-1} f'_+(-\xi_\eps,\eps k) = \rmi k \theta.
\]
The behaviour of the Jost solution~$f_-$ on the right-half line can be analyzed similarly.
\end{proof}


\section{Convergence of the auxiliary operator family $\wt S_\eps$}\label{sec:aux-fam}

\subsection{A general approach}
In what follows, we denote by $\chi_\eps$ the characteristic function of the interval~$[-x_\eps,x_\eps]$ and consider an auxiliary family~$\wt S_\eps$ of Schr\"odinger operators on the line defined via
\[
    \wt S_\eps := - \frac{\rmd^2}{\rmd x^2}
        + \frac1{\eps^2}\,V\Bigl(\frac{x}\eps\Bigr)\chi_\eps(x).
\]
Based on the results of the papers~\cite{AlbeverioCacciapuotiFinco:2007,CacciapuotiExner:2007, GolovatyHryniv:2010,GolovatyManko:2009} we expect that the family~$\wt
S_\eps$ converges to a limit $S_0$ as $\eps \to0$ in the norm resolvent sense and that~$S_0$ is
$S_-\oplus S_+$ in the non-resonant case and $S(\theta)$, with $\theta$ of~\eqref{eq:intr.theta} or~\eqref{eq:theta}, in the
resonant case.

To explain the main idea behind the constructions that follow, we fix a nonreal $k^2$, take an arbitrary
$f\in L_2(\bR)$, and set $y_0:=(S_0-k^2)^{-1}f$ and $\wt y_\eps:=(\wt S_\eps - k^2)^{-1}f$. To show that
$\wt y_\eps \to y_0$ as $\eps\to0$, we construct an auxiliary function~$y_\eps$ in $\dom \wt S_\eps$ such
that $q_\eps:=y_\eps - y_0 \to 0$ and $r_\eps:=(\wt S_\eps - k^2) y_\eps - f\to0$ in the $L_2$-norm as
$\eps\to0$; it then follows that
\[
    \wt y_\eps - y_0 = (y_\eps - y_0) - (y_\eps - \wt y_\eps) = q_\eps - (\wt S_\eps - k^2)^{-1} r_\eps
    \to0.
\]
A more careful analysis shows that the convergence of~$\wt y_\eps$ to~$y_0$ is uniform in~$f$ on bounded sets, thus establishing the norm resolvent convergence of $\wt S_\eps$ to $S_0$.

We start the construction of the approximation~$y_\eps$ by observing that the function~$\wt y_\eps=(\wt S_\eps - k^2)^{-1}f$ solves the equations
\[
    -y'' = k^2 y + f
\]
for $|x|> x_\eps$ and
\[
    -y'' + \eps^{-2} V(x/\eps) y = k^2 y + f
\]
for $|x| < x_\eps$. Setting $\wt y_\eps(x):= w(x/\eps)$  in the ``fast variable'' domain~$|x|<x_\eps$, we see that $w$ must solve the equation
\[
    - w''(\xi) + V(\xi) w(\xi) = \eps^2 k^2 w(\xi) + \eps^2 f(\eps\xi)
\]
for $\xi \in (-\xi_\eps,\xi_\eps)$, see Figure~\ref{cutoff}. A general solution to this equation has the form $w_\eps + u_\eps$, where
\[
    w_\eps = a_-(\eps) f_-(\,\cdot\,,\eps k)+ a_+(\eps) f_+(\,\cdot\,,\eps k),
\]
with $a_\pm(\eps)$ arbitrary complex numbers, is a solution to the homogeneous problem and
\begin{eqnarray}
    &u_\eps(\xi) := \dfrac{\eps^2 f_+(\xi,\eps k)}{D(\eps k)}
            \int_{-\xi_\eps}^\xi f_-(\eta,\eps k) f(\eps \eta)\,\rmd \eta\nonumber \\
             &\lefteqn{
             +\displaystyle
            \frac{\eps^2 f_-(\xi,\eps k)}{D(\eps k)}
            \int_{\xi}^{\xi_\eps} f_+(\eta,\eps k) f(\eps \eta)\,\rmd \eta}\nonumber
\end{eqnarray}
is a particular solution.

\begin{figure}[t]\label{fig:1}
\begin{center}
\includegraphics[scale=0.45]{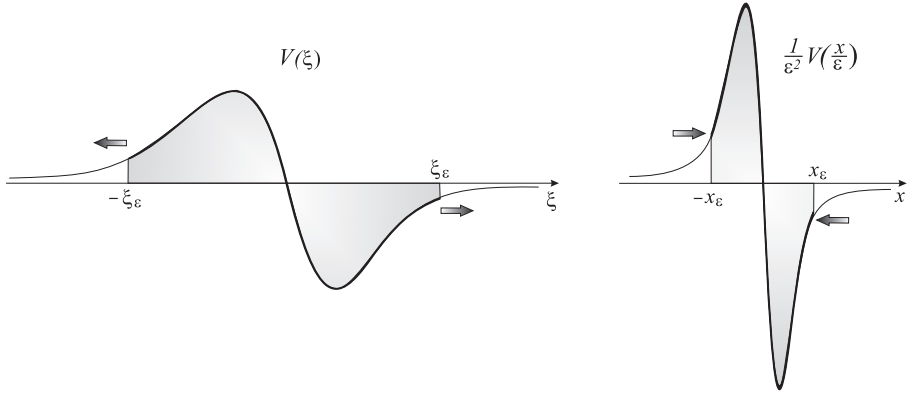}
\end{center}\caption{\label{cutoff} Potentials  in $\xi$ and $x$ coordinates}
\end{figure}

It is therefore natural to take a function
\begin{equation*}\label{eq:zeps}
    z_\eps(x) := y_0(x) [1-\chi_\eps(x)] + [w_\eps(x/\eps)+u_\eps(x/\eps)]\chi_\eps(x)
\end{equation*}
equal to $y_0$ for $|x|> x_\eps$ and to $w_\eps(x/\eps)+u_\eps(x/\eps)$ for $|x|<x_\eps$ as a leading term of the asymptotic expansion of~$y_\eps$ as~$\eps\to0$. The function~$z_\eps$ belongs to $W_2^2$ outside the points $x=\pm x_\eps$ and  has jump discontinuities at these points; namely, denoting by $[g]_a$ the jump of a function $g$ at a point $x = a$, we see that
\begin{equation}
\label{eq:jumpsZ}
  \begin{array}{l}
   \bigl|[z_\eps]_{\pm x_\eps}\bigr|
        =\bigl|y_0(\pm\xp)- w_\eps(\pm\xip)-u_\eps(\pm\xip)\bigr|,\\
      \bigl|[z_\eps']_{\pm x_\eps}\bigr|
        =\bigl|y_0'(\pm\xp)- \eps^{-1} w_\eps'(\pm \xip)-\eps^{-1}u_\eps'(\pm\xip)\bigr|.
  \end{array}
\end{equation}

Our task is to choose the coefficients $a_+(\eps)$ and $a_-(\eps)$ defining~$w_\eps$ in such a way that the values of the right hand sides of~\eqref{eq:jumpsZ}  vanish as $\eps\to0$.
We first show that the particular solution $u_\eps$ does not contribute significantly to these values. Set
$$  \Delta_\eps:=[\xim,\xip],\quad
    \Delta_\eps^-:=[\xim,0], \quad
    \Delta_\eps^+:=[0,\xip].$$

\begin{lem}\label{lem:ueps}
The following holds as $\eps \to 0$:
\begin{eqnarray}\label{eq:UepsEst}
    &&\phantom{\eps^{-1}|}|u_\eps(\pm \xi_\eps)|= O(x_\eps^{1/2})\|f\|,\\
    &&\eps^{-1}\left|u_\eps'(\pm\xi_\eps)\right| = O(x_\eps^{1/2})\|f\|,  \label{eq:U'epsEst}
\end{eqnarray}
and, in addition,
    $\|u_\eps(\,\cdot\,/\eps )\chi_\eps\|= O(x_\eps)\|f\|$.
\end{lem}

\begin{proof}
First, observe that
\begin{eqnarray}\nonumber
    &&\max_{\Delta_\eps^+}|u_\eps|
    \leq
    \frac{2\eps^2}{|D(\eps k)|}
    \max_{\Delta_\eps^+}|f_+(\,\cdot\,,\eps k)|\,\max_{\Delta_\eps}|f_-(\,\cdot\,,\eps k)|
    \int_{\xim}^{\xi_\eps} |f(\eps\eta)|\,\rmd\eta,   \\\nonumber
     &&\max_{\Delta_\eps^-}|u_\eps|\leq
    \frac{2\eps^2}{|D(\eps k)|}
    \max_{\Delta_\eps}|f_+(\,\cdot\,,\eps k)|\,\max_{\Delta_\eps^-}|f_-(\,\cdot\,,\eps k)|
    \int_{\xim}^{\xip} |f(\eps\eta)|\,\rmd\eta
\end{eqnarray}
and that
  \begin{equation}\label{eq:intf(epstau)}
  \int_{\xim}^{\xip} |f(\eps\eta)|\,\rmd\eta\leq  \sqrt{2}\,\eps^{-1/2}\xip^{1/2}\|f\|
  \end{equation}
by the Cauchy--Bunyakowski inequality.

In the non-resonant case, $D(0)\neq 0$, and by virtue of Proposition~\ref{pro:fpm} we have
\[  
    \max_{\Delta_\eps}|u_\eps|
    \leq c_1 \eps^{3/2}(1+\xip)\,\xip^{1/2}\|f\|\leq 2c_1 \xp^{3/2}\|f\|
\]  
for all sufficiently small $\eps$ and some $c_1>0$ independent of~$\eps$.
In the resonant case, we use \eqref{eq:limD(ek)/e} and the fact that, by Lemma~\ref{lem:f+-asymptoticsResonance}, the Jost solutions $f_\pm(\,\cdot\,,\eps k)$ remain bounded over $\Delta_\eps$ uniformly in~$\eps$  to find $c_2$ independent of $\eps$ such that
\[  
    \max_{\Delta_\eps}|u_\eps|\leq c_2 \xp^{1/2}\|f\|.
\]  
Since the last inequality holds also in the non-resonant case for small enough~$\eps$, we readily deduce \eqref{eq:UepsEst} and estimate the $L_2$-norm of  $u_\eps(\,\cdot\,/\eps)\chi_\eps$ via
\[
    \int_{-x_\eps}^{x_\eps} |u_\eps(x/\eps)|^2\,\rmd x =
        \eps \int_{-\xi_\eps}^{\xi_\eps} |u_\eps(\xi)|^2\,\rmd\xi
                 \le  2c_2^2\,\eps x_\eps \xip \,\|f\|^2 = O(x^2_\eps)\|f\|^2
\]
as $\eps\to0$.

Next we compute the derivative
\[
    u_\eps'(\xi_\eps) = \frac{\eps^2 f_+'(\xip,\eps k)}{D(\eps k)}
            \int_{\xim}^{\xip} f_-(\eta,\eps k) f(\eps \eta)\,\rmd \eta
\]
and obtain
\begin{eqnarray}\nonumber
    &\displaystyle\eps^{-1}|u_\eps'(\xi_\eps)|\le \frac{\eps}{|D(\eps k)|}\,|f_+'(\xip,\eps k)|\,\max_{\Delta_\eps}|f_-(\xi,\eps k)|\int_{\xim}^{\xip} |f(\eps\eta)|\,\rmd\eta&\\\nonumber
    &\lefteqn{\le  \frac{c_3 |k|\eps\xp^{1/2}}{|D(\eps k)|}\,\max_{\Delta_\eps}|f_-(\xi,\eps k)|\,\|f\|}&
\end{eqnarray}
due to \eqref{eq:f'+-xi_eps} and \eqref{eq:intf(epstau)}. Applying Proposition~\ref{pro:fpm}
in the non-resonant case and Lemma~\ref{lem:f+-asymptoticsResonance} along with \eqref{eq:limD(ek)/e} in the resonant one, we find that
\[
    \eps^{-1}|u_\eps'(\xi_\eps^{1/2})|\le c_4\xp\|f\|.
\]
Similar considerations yield the estimate for $u_\eps'(\xim)$.
\end{proof}

Next we need to choose the coefficients  $a_\pm(\eps)$ in a proper way. This will be done separately in the non-resonant and resonant cases.

\subsection{A special solution in the non-resonant case}\label{ssec:aux.nres}
Observing that $y_0(\pm x_\eps) \to 0$ and that $y'_0(\pm x_\eps)\to y'_0(0\pm)$ as $\eps \to0$ and recalling that the values $|f'_\pm(\ximp,\eps k)|$ converge to the positive number $|D(0)|$ by \eqref{f'-+xi_epsD0}, one concludes that a suitable choice of $a_\pm(\eps)$ might be
\[
    a_-(\eps) :=  \eps y_0'(x_\eps)/f'_-(\xi_\eps,\eps k),\qquad
    a_+(\eps) :=  \eps y_0'(-x_\eps)/f'_+(-\xi_\eps,\eps k).
\]
And indeed, the relations of Proposition~\ref{pro:fpm} and Lemma~\ref{lem:f+-asymptotics}  yield then the asymptotics
\begin{eqnarray}\nonumber
    &\phantom{\eps^{-1}}w_\eps(\xipm)
\displaystyle =\frac{\eps f_-(\pm \xi_\eps,\eps k)}{f'_-(\xi_\eps,\eps k)}\, y_0'(x_\eps)
                    + \frac{\eps f_+(\pm \xi_\eps,\eps k)}{f'_+(-\xi_\eps,\eps k)}\,y_0'(-x_\eps)&
                    \\\nonumber
               &\lefteqn{=O(x_\eps)\cdot[y_0'(x_\eps) + y_0'(-x_\eps)],}&\\\nonumber
    &\eps^{-1} w'_\eps(\xipm)
            \displaystyle =\frac{f'_-(\pm \xi_\eps,\eps k)}{f'_-(\xi_\eps,\eps k)}\, y_0'(x_\eps)
                    + \frac{f'_+(\pm \xi_\eps,\eps k)}{f'_+(-\xi_\eps,\eps k)}\, y_0'(-x_\eps)&\\
            &\lefteqn{ = y_0'(\pm x_\eps) + O(\eps) \cdot y_0'(\mp x_\eps).}&\nonumber
\end{eqnarray}
Since the resolvent of $S_0$ acts continuously from $L_2(\bR)$ into $W_2^2(\bR\setminus\{0\})$, there
exists a constant $c_1$ independent of~$f$ such that
\[
    \|y_0\|_{W_2^2(\bR\setminus\{0\})} \le c_1 \|f\|;
\]
thus the Sobolev embedding theorem yields the estimates
\begin{equation}\label{eq:y0-sobolev}
    \sup_{x}(|y_0(x)|+|y_0'(x)|) \le c_2 \|y_0\|_{W_2^2(\bR\setminus\{0\})} \le c_3 \|f\|
\end{equation}
as well as the relations
\begin{eqnarray}\nonumber
    &&|y_0(\pm x_\eps)|
        = \Bigl|\int_0^{\pm x_\eps} y'_0(t)\,\rmd t\Bigr| \le x_\eps^{1/2}\|y'_0\|
        \le c_4 x_\eps^{1/2} \|f\|,\\\nonumber
     && |y_0'(\pm x_\eps)-y_0'(0\pm)|
        = \Bigl|\int_0^{\pm x_\eps} y''_0(t)\,\rmd t\Bigr| \le x_\eps^{1/2}\|y''_0\|
        \le c_5 x_\eps^{1/2} \|f\|.
\end{eqnarray}

Returning now to formulae~\eqref{eq:jumpsZ} and combining the above estimates with Lemma~\ref{lem:ueps}, we get
\begin{eqnarray}\nonumber
    &&\bigl| [z_\eps]_{\pm x_\eps}\bigr|
                    \le |y_0(\pm x_\eps)| + |w_\eps( \xipm)|+ |u_\eps( \xipm)|=  O(x_\eps^{1/2}) \cdot \|f\|,\\\nonumber
    &&\bigl| [z'_\eps]_{\pm x_\eps}\bigr|  \le |y'_0(\pm x_\eps) - \eps^{-1} w'_\eps(\xipm) |+\eps^{-1} |u'_\eps( \xipm) |
        =   O(\xp^{1/2}) \cdot \|f\|.
\end{eqnarray}

We shall also need an estimate on the $L_2$-norm of $w_\eps(\,\cdot\,/\eps)\chi_\eps$.
Since $a_\pm(\eps)=O(\eps)\|f\|$ and
\[
    \int_{-x_\eps}^{x_\eps} |f_{\pm}(x/\eps,\eps k)|^2\,\rmd x
            = \eps \int_{\xim}^{\xip} |f_{\pm}(\xi,\eps k)|^2\,\rmd\xi
            \le c_6 \eps (1+\xip)^3,
\]
we find that, as $\eps\to0$,
\[
    \|w_\eps(\,\cdot\,/\eps)\chi_\eps\| = O(x_\eps^{3/2})\,\|f\|.
\]


\subsection{A special solution in the resonant case}\label{ssec:aux.res}
In the resonant case, the limit operator~$S_0$ is expected to be $S(\theta)$, with $\theta$
of~\eqref{eq:theta}. The main difference with the non-resonant case is that the zero energy Jost solutions $f_-$ and $f_+$ are now linearly dependent and thus follow quite a different asymptotics.

A possible choice for $a_\pm(\eps)$ could be the one resulting in the equalities
\[
w_\eps(\xim)=y_0(-x_\eps),\qquad \eps^{-1}w'_\eps(\xim)=y_0'(-x_\eps).
\]
However, we have found it more convenient to take the limiting values $a_\pm(0)$ instead of such~$a_\pm(\eps)$ for all $\eps>0$, i.e., to set
\begin{eqnarray}\nonumber
    &&a_-(\eps) \equiv  \frac{-1}{\rmi k(\theta+\theta^{-1})}\,
            \bigl[y'_0(0-)\theta^{-1} - \rmi k \theta y_0(0-)\bigr], \\\nonumber
    &&a_+(\eps) \equiv \frac{1}{\rmi k(\theta + \theta^{-1})}\,
            \bigl[\rmi k y_0(0-) + y'_0(0-)\bigr].
\end{eqnarray}
Then, as $\eps\to0$,
\begin{eqnarray}\nonumber
    &&w_\eps(-\xi_\eps)  = y_0(0-)
        \frac{ \theta f_-(-\xi_\eps,\eps k) + f_+(-\xi_\eps,\eps k)}{\theta+\theta^{-1}}\\\nonumber
        && \qquad\qquad+ y'_0(0-)
        \frac{-\theta^{-1} f_-(-\xi_\eps,\eps k) + f_+(-\xi_\eps,\eps k)}
            {\rmi k(\theta+\theta^{-1})}\\\nonumber
        &&\qquad\qquad=y_0(0-)(1+ o(1)) + y'_0(0-)\cdot o(1) \\\nonumber
        &&\qquad\qquad= y_0(-x_\eps) + o(1)\|f\|
\end{eqnarray}
and, similarly,
\begin{eqnarray}\nonumber
    &&\eps^{-1}w'_\eps(-\xi_\eps)  = y_0(0-)
        \frac{ \theta f'_-(-\xi_\eps,\eps k) + f'_+(-\xi_\eps,\eps k)}{\eps(\theta+\theta^{-1})}\\\nonumber
        &&\qquad\qquad\quad + y'_0(0-)
        \frac{-\theta^{-1} f'_-(-\xi_\eps,\eps k) + f'_+(-\xi_\eps,\eps k)}
            {\rmi \eps k(\theta+\theta^{-1})}\\\nonumber
        &&\qquad\qquad\quad= y_0(0-)\cdot o(1) + y'_0(0-)(1+ o(1)) \\\nonumber
        &&\qquad\qquad\quad= y'_0(-x_\eps) + o(1)\|f\|.
\end{eqnarray}
Next, recalling the relations  $y_0(0-)=\theta^{-1} y_0(0+)$ and  $y'_0(0-) = \theta y'_0(0+)$, we see that
\begin{eqnarray}\nonumber
    &&w_\eps(\xi_\eps)  = y_0(0+)
        \frac{ f_-(\xi_\eps,\eps k) + \theta^{-1}f_+(\xi_\eps,\eps k)}{\theta+\theta^{-1}}\\\nonumber
        &&\qquad\quad + y'_0(0+)
        \frac{- f_-(\xi_\eps,\eps k) + \theta f_+(\xi_\eps,\eps k)}
            {\rmi k(\theta+\theta^{-1})}\\\nonumber
        &&\qquad\quad =y_0(0+)(1+ o(1)) + y'_0(0+)\cdot o(1) \\\nonumber
        &&\qquad\quad = y_0(x_\eps) + o(1)\|f\|
\end{eqnarray}
and
\begin{eqnarray}\nonumber
    &&\eps^{-1}w'_\eps(\xi_\eps)  = y_0(0+)
        \frac{ f'_-(\xi_\eps,\eps k) + \theta^{-1}f'_+(\xi_\eps,\eps k)}{\eps(\theta+\theta^{-1})}\\\nonumber
        &&\qquad\qquad\; + y'_0(0+)
        \frac{- f'_-(\xi_\eps,\eps k) + \theta f'_+(\xi_\eps,\eps k)}
            {\rmi \eps k(\theta+\theta^{-1})}\\\nonumber
        &&\qquad\qquad\; = y_0(0+)\cdot o(1) + y'_0(0+)(1+ o(1)) \\\nonumber
        &&\qquad\qquad\; = y'_0(x_\eps) + o(1)\cdot\|f\|.
\end{eqnarray}
Combining the above relations, we conclude that the jumps of the function $z_\eps$ and its derivative at the points $\pm x_\eps$ are small, namely that
\begin{eqnarray}\nonumber
     &&[z_\eps]_{\pm x_\eps}  = o(1) \cdot \|f\|,\\\nonumber
     &&[z'_\eps]_{\pm x_\eps}  = o(1) \cdot \|f\|
\end{eqnarray}
as $\eps \to 0$.

Finally, using the fact that by Lemma~\ref{lem:f+-asymptoticsResonance} the Jost solutions $f_{\pm}(\,\cdot\,,\eps k)$ remain uniformly bounded over $(-\xi_\eps,\xi_\eps)$ as $\eps\to0$, we arrive at the estimate
\[
    \int_{-x_\eps}^{x_\eps} |f_\pm(x/\eps,\eps k)|^2\,\rmd x
            = \eps \int_{-\xi_\eps}^{\xi_\eps} |f_\pm(\xi,\eps k)|^2\,\rmd\xi
            \le c_7 \eps  \xi_\eps = c_7 x_\eps
\]
for all small $\eps$. Observing that $a_\pm(\eps)=O(1)\|f\|$ as $\eps\to0$, we get the estimate
\[
    \|w_\eps(\,\cdot\,/\eps)\chi_\eps\| = O(x_\eps)\,\|f\|.
\]


\subsection{Convergence of $\widetilde S_\eps$}\label{ssec:aux.conv}

We are now in a position to prove the main result of this section.

\begin{thm}\label{thm:conv-tildeS}
Denote by $S_0$ the operator $S_-\oplus S_+$ in the non-resonant case and the operator $S(\theta)$ with
$\theta:= f_-/f_+$ in the resonant case. Then the operator family $\widetilde S_\eps$ converges in the norm resolvent sense as $\eps\to 0$
to~$S_0$.
\end{thm}

\begin{proof}
Fix $k^2 \in \mathbb{C}\setminus\mathbb{R}$, take an arbitrary  $f \in L_2(\bR)$, and set $y_0:=(S_0 - k^2)^{-1}f$. As the
first step, we construct a function $q_\eps\in L_2(\bR)$ with $\norm{q_\eps} = o(\norm{f})$ as $\eps\to0$
such that the function $y_\eps:= y_0 + q_\eps$ belongs to $\dom \widetilde S_\eps$ and, as $\eps\to0$, the
relation
\begin{equation*}\label{SepsEst}
    \bigl\|(\widetilde S_\eps-k^2)y_\eps-f\bigr\| = o(\norm{f})
\end{equation*}
holds. We then show that this yields the required norm resolvent convergence.

\begin{figure}[ht]\label{fig:2}
\begin{center}
\includegraphics[scale=1]{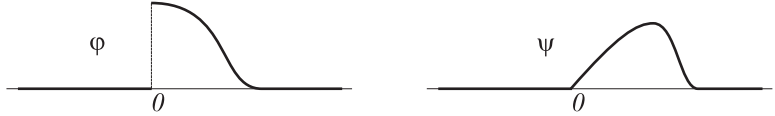}
\end{center}\caption{\label{Jump} Functions with the prescribed jumps at the origin}
\end{figure}

To construct such a $y_\eps$, we shall modify $z_\eps$ at the points $\pm x_\eps$ to eliminate the jumps.
To this end let us introduce functions $\varphi$ and $\psi$ as in Figure~\ref{Jump} that are smooth
outside the origin, have compact supports contained in $[0,\infty)$, and have the prescribed jumps
$[\varphi]_0=1$, $[\varphi']_0=0$ and $[\psi]_0=0$, $[\psi']_0=1$. Set
\begin{eqnarray}\nonumber
    &&\zeta_\eps(x) = [z_\eps]_{-x_\eps}\,\varphi(-x-x_\eps)
                -[z'_\eps]_{-x_\eps}\,\psi(-x-x_\eps)\\\label{Zeta}
        &&\qquad\qquad\qquad-[z_\eps]_{x_\eps}\,
            \varphi(x-x_\eps)-[z'_\eps]_{x_\eps}\,\psi(x-x_\eps);\nonumber
\end{eqnarray}
then $\zeta_\eps =0$ on~$(-x_\eps,x_\eps)$ and, in view of the estimates on~$[z_\eps]_{\pm x_\eps}$ and
$[z'_\eps]_{\pm x_\eps}$ of Subsections~\ref{ssec:aux.nres} and \ref{ssec:aux.res},
\begin{equation}\label{ZetaEstim}
    \max_{|x|>x_\eps}\left|\zeta^{(k)}_\eps(x)\right| =   o(\|f\|)
\end{equation}
as $\eps\to0$ for $k=0,1,2$.

Clearly, the function $y_\eps:= z_\eps+\zeta_\eps$ is continuous on~$\bR$ along with its derivative and
belongs to $W_2^2(\bR) = \dom S_\eps$. Observe that $y_\eps=y_0$ for $|x|$ large enough; more exactly, we
have
\begin{equation*}\label{yeps}
    y_\eps =  y_0 +q_\eps=(S_0-k^2)^{-1}f+q_\eps
\end{equation*}
with
\[
    q_\eps:=[w_\eps(\,\cdot\,/\eps) + u_\eps(\,\cdot\,/\eps) - y_0]\chi_\eps + \zeta_\eps
\]
of compact support. Taking into account Lemma~\ref{lem:ueps}, the estimates on $\|w_\eps(\,\cdot\,/\eps)\chi_\eps\|$ derived above, and relations \eqref{eq:y0-sobolev} and~\eqref{ZetaEstim}, we see that the $L_2$-norm of the function~$q_\eps$ is~$o(\|f\|)$.

Next we calculate the function~$(\wt S_\eps-k^2) y_\eps$ for $|x|< x_\eps$ and $|x|>x_\eps$. If
$|x|<x_\eps$, then $\zeta_\eps(x)=0$, whence $y_\eps(x)=z_\eps(x)$ and
\[
 (\widetilde S_\eps-k^2)y_\eps(x)
    = \eps^{-2} \bigl[-z_\eps''(\tfrac{x}{\eps})
            +V(\tfrac{x}{\eps})z_\eps(\tfrac{x}{\eps})\bigr] - k^2 z_\eps(\tfrac{x}{\eps})\\
    = f(x).
\]
If $|x|>x_\eps$, then $y_\eps =  y_0 +\zeta_\eps$, and hence
\begin{eqnarray}\nonumber
     &&(\widetilde S_\eps-k^2)y_\eps(x)
        = \left(-\tfrac{\rmd^2}{\rmd x^2} - k^2\right)(y_0 + \zeta_\eps)(x)\\
        &&\qquad\qquad\qquad\qquad= f(x) - \zeta''_\eps(x) - k^2\zeta_\eps(x).\nonumber
\end{eqnarray}
Therefore, $(\widetilde S_\eps-k^2)y_\eps= f + r_\eps$ with the remainder
\begin{equation*}\label{RepsNonRes}
 r_\eps(x) = -\zeta''_\eps(x)- k^2\zeta_\eps(x).
\end{equation*}
Using~\eqref{ZetaEstim}, we find that, as $\eps\to0$,
\[
    \|r_\eps\| = o(\|f\|).
\]

Finally, since $\|(\wt S_\eps-k^2)^{-1}\| \le |\myIm k^2|^{-1}$, it follows that
\begin{eqnarray}\nonumber
   && \|(\widetilde S_\eps-k^2)^{-1}f-(S_0 -k^2)^{-1}f\|
        = \|\wt y_\eps - y_0\|\\\nonumber
         &&\qquad\qquad\qquad\qquad\leq  \|q_\eps\|+\|(\wt S_\eps-k^2)^{-1}r_\eps\|
         = o(\|f\|)
\end{eqnarray}
as $\eps\to0$, and the proof is complete.
\end{proof}

\section{Jost solutions and scattering coefficients for the operator~$\wt S_\eps$}\label{sec:Jost}

Our next task is to show that the resolvents of the operators $\wt S_\eps$ and $S_\eps$ get closer as $\eps$ tends to zero. To do this, we shall need more information on the resolvent of the operator~$\wt S_\eps$.

We denote by $\widetilde f_\pm(\,\cdot\,,\eps,k)$ the Jost solutions of the operator~$\wt S_\eps$; then the
resolvent $\wt R_\eps(k):=(\wt S_\eps - k^2)^{-1}$ is an integral operator with kernel equal to the Green
function $\wt G_\eps(x,y,k)$,
\[
    \wt G_\eps(x,y,k) = \frac1{\wt D_\eps(k)}
    \left\{
      \begin{array}{ll}
        \wt f_+(x,\eps,k) \wt f_-(y, \eps,k), & x > y, \\
        \wt f_-(x,\eps,k) \wt f_+(y, \eps,k), & x < y,
      \end{array}
    \right.
\]
where $\wt D_\eps(k)$ is the Wronskian of the Jost solutions $\wt f_+(\,\cdot\,,\eps,k)$ and $\wt
f_-(\,\cdot\,,\eps,k)$, i.e.,
\[
    \wt D_\eps(k) =   {\wt f}_+(\,\cdot\,,\eps,k) \wt f_- ^{\,\prime}(\,\cdot\,,\eps,k)
                      - \wt f_+^{\,\prime} (\,\cdot\,,\eps,k) \wt f_-(\,\cdot\,,\eps,k).
\]


In this section, we shall construct the Jost solutions $\wt f_\pm(\,\cdot\,,\eps,k)$ of the operator~$\wt
S_\eps$ for a fixed $k\in\bC$ with $\myIm k >0$ and $\myRe k>0$, and then study their behaviour as
$\eps\to0$. The analysis of the right and left Jost solutions is quite similar, and only the case of
$\widetilde f_+$ will be investigated in detail. As a by-product, we establish some estimates on $\wt f_\pm(\,\cdot\,,\eps,k)$ that will essentially be used in the next section to prove the main result.

\subsection{Jost solutions: a general construction}
Since the potential of $\wt S_\eps$ vanishes for $|x|>x_\eps$, we consider separately three regions:
$x>x_\eps$, $|x|<x_\eps$, and $x<-x_\eps$.

\textbf{Region 1:} for $x>x_\eps$ we clearly have
\[
    \widetilde f_+(x,\eps,k) = \rme^{\rmi kx}.
\]

\textbf{Region 2:} for $|x|<x_\eps$ the Jost solution satisfies the equation
\[
    -y'' + \eps^{-2}V(x/\eps)y =  k^2 y
\]
and thus is equal to
\begin{equation*}\label{eq:prox.reg2}
    \widetilde f_+(x,\eps,k) = c_\eps^+ f_+(x/\eps,\eps k) + c_\eps^- f_-(x/\eps,\eps k),
\end{equation*}
where $f_\pm$ are the Jost solutions of~\eqref{eq:Jost.de}, i.e., of the Schr\"odinger operator~$S_1$. Continuity of $\widetilde f_+$ and its derivative at the point~$x=x_\eps$ results in the system
\[\left\{
    \begin{array}{l}
       c_\eps^+ f_+(\xi_\eps,\eps k)  + c_\eps^- f_-(\xi_\eps,\eps k)  = \rme^{\rmi kx_\eps},\\
       c_\eps^+ f'_+(\xi_\eps,\eps k) + c_\eps^- f'_-(\xi_\eps,\eps k) = \rmi\eps k\rme^{\rmi
            kx_\eps}.
    \end{array}
  \right.
 \]
Recall that
\[
    D(k) := f_+(\xi,k)f'_-(\xi,k)-f'_+(\xi,k)f_-(\xi,k)
\]
is the Wronskian of the Jost solutions~$f_+$ and $f_-$ and observe that $D(\eps k)\ne0$ for every positive~$\eps$ since otherwise $f_+(\,\cdot\,,\eps k)$ and $f_-(\,\cdot\,,\eps k)$ would be linearly dependent and thus $\eps^2k^2 \in \bC^+$ would be an eigenvalue of the self-adjoint operator~$S_1$. Solving the above system for $c_\eps^\pm$, one gets
\begin{equation*}\label{eq:prox.cpm}
  \begin{array}{l}
      c_\eps^+ = \dfrac{\rme^{\rmi kx_\eps}}{D(\eps k)}\,
                    \biggl[\, f'_-(\xi_\eps,\eps k)-\rmi \eps k f_-(\xi_\eps,\eps k) \biggr];\\
   c_\eps^- = \dfrac{\rme^{\rmi kx_\eps}}{D(\eps k)}\,\,
                    \biggl[\,\rmi \eps k f_+(\xi_\eps,\eps k)- f'_+(\xi_\eps,\eps k) \biggr].
  \end{array}
\end{equation*}

In the non-resonant case, $D(0)\ne0$, and thus $|D(\eps k)|\ge |D(0)|/2$ for all $\eps$ small enough by the continuity of~$D$. Then $c_\eps^- \to0$ by~\eqref{eq:f+'0}, while $c_\eps^+\to1$ by Lemma~\ref{lem:f+-asymptotics}.
In the resonant case, we combine this lemma with  relation~\eqref{eq:limD(ek)/e} and get the same convergence $c_\eps^-\to0$ and $c_\eps^+\to1$ as $\eps\to0$.

\textbf{Region 3:} for $x<-x_\eps$,
\begin{equation*}\label{eq:prox.reg3}
    \widetilde f_+(x,\eps,k) = a_\eps^+\rme^{\rmi kx} + b_\eps^+ \rme^{-\rmi kx}
\end{equation*}
for some coefficients $a_\eps^+$ and $b_\eps^+$. Continuity of the Jost solution $\wt f_+$ and its
derivative at $x=-x_\eps$ gives the linear system for $a_\eps^+$ and $b_\eps^+$:
\[\left\{
    \begin{array}{l}
      \phantom{\rmi k(i} a_\eps^+ \rme^{-\rmi kx_\eps}  + b_\eps^+ \rme^{\rmi kx_\eps}  = \wt f_+(-x_\eps+0,\eps,
            k),\\
            \rmi k(a_\eps^+ \rme^{-\rmi kx_\eps}  - b_\eps^+ \rme^{\rmi kx_\eps})
                    = {\wt f_+}^{\,\prime}(-x_\eps+0,\eps, k),
    \end{array}
  \right.
\]
solving which we find that
\begin{eqnarray}\nonumber
    &&a_\eps^+ = \dfrac{\rme^{\rmi kx_\eps}}{2\rmi k}\,
                    \biggl[\,\rmi k \wt f_+(-x_\eps+0,\eps,k) + \wt f'_+(-x_\eps+0,\eps, k)\biggr];\\
    &&b_\eps^+ = \dfrac{\rme^{-\rmi kx_\eps}}{2\rmi k}\,
                    \biggl[\,\rmi k \wt f_+(-x_\eps+0,\eps,k) - \wt f'_+(-x_\eps+0,\eps, k)\biggr].\nonumber
\end{eqnarray}

\begin{rem}
Calculating the Wronskian $\wt D_\eps(k)$ of the Jost solutions $\wt f_\pm$ at a point $x$ to the left of~$-x_\eps$, one
immediately sees that $\wt D_\eps(k) = - 2\rmi k a_\eps^+$.
\end{rem}

To investigate the behaviour of the coefficients $a_\eps^+$ and $b_\eps^+$, we have to
consider the resonant and non-resonant cases separately.

\subsection{Scattering coefficients in the non-resonant case} First we note that, as $\eps\to0$,
\[
    \wt f_+(-x_\eps+0,\eps,k) = c_\eps^+ f_+(-\xi_\eps,\eps k) + c_\eps^- f_-(-\xi_\eps,\eps k)
        =O(\xi_\eps),
\]
and that, in view of \eqref{f'-+xi_epsD0},
\[
    \eps\wt f'_+(-x_\eps+0,\eps,k) = c_\eps^+ f'_+(-\xi_\eps,\eps k) + c_\eps^- f'_-(-\xi_\eps,\eps k)
        \to - D(0).
\]
It therefore follows that
\[
    \eps a_\eps^+ \to - D(0),
    \qquad
    \eps b_\eps^+ \to  D(0)
\]
as $\eps \to 0$.
The above analysis remains in force for real $k$; therefore, one gets the following result.

\begin{cor}\label{cor:prox.nres}
Assume that the potential~$V$ is non-resonant. Then, for every nonzero $k$ with $\myIm k\ge0$, the
reflection~$\wt r_\eps(k):= b^+_\eps/a^+_\eps$ and transmission~$\wt t_\eps(k):=1/a^+_\eps$ coefficients of the
Schr\"odinger operator $\wt S_\eps$ satisfy the asymptotic relations $\wt r_\eps(k)\to-1$ and $\wt
t_\eps(k)\to 0$ as $\eps\to0$.
\end{cor}

This can be compared with the analogous result
in~\cite{GolovatyHryniv:2010,GolovatyManko:2009,Manko:2009} proved for the operator
family~$S_\eps$ in the case where the support of~$V$ is contained in~$[-1,1]$.


\subsection{Scattering coefficients  in the resonant case} As $\eps\to0$, we have
\[
    \wt f_+(-x_\eps+0,\eps,k) = c_\eps^+ f_+(-\xi_\eps,\eps k) + c_\eps^- f_-(-\xi_\eps,\eps k)
        \to \theta^{-1},
\]
by \eqref{eq:f+-xi_eps} and \eqref{eq:f+-res}, while \eqref{eq:f'+-xi_eps}  and \eqref{eq:f'+-res} yield
\[
    \wt f'_+(-x_\eps+0,\eps,k) = c_\eps^+ \eps^{-1}f'_+(-\xi_\eps,\eps k)
        + c_\eps^- \eps^{-1}f'_-(-\xi_\eps,\eps k)
        \to \rmi k \theta.
\]
It therefore follows that, as $\eps \to 0$,
\[
    a_\eps^+ \to \tfrac12[\theta^{-1} + \theta],
    \qquad
    b_\eps^+ \to \tfrac12[\theta^{-1} - \theta].
\]
The above analysis remains in force for real $k$; therefore, one arrives at the following result.

\begin{cor}\label{cor:prox.res}
Assume that the potential~$V$ is resonant and that $\theta$ is given by \eqref{eq:intr.theta} (or, equivalently, by~\eqref{eq:theta}). Then, for every nonzero $k$
with $\myRe k \ge0$, the reflection~$\wt r_\eps(k):= b^+_\eps/a^+_\eps$ and transmission~$\wt
t_\eps(k):=1/a^+_\eps$ coefficients of the Schr\"odinger operator $\wt S_\eps$ satisfy the asymptotic
relations
\[
\wt r_\eps(k)\to \frac{1-\theta^2}{1+\theta^2},\qquad \wt t_\eps(k)\to \frac{2\theta}{1+\theta^2}\quad \mbox{ as }
\eps\to0.
\]

\end{cor}

We note that the above limits are the reflection and transmission coefficients of the Schr\"odinger operator $S(\theta)$ and that they coincide with the value at $k=0$ of the reflection and transmission coefficients for the Sch\"odinger operator~$S_1$, see~\cite{Klaus:1988}. This can also be compared with the analogous result in~\cite{GolovatyHryniv:2010,GolovatyManko:2009,Manko:2009} proved for the
operator family~$S_\eps$ in the case where the support of~$V$ is contained in~$[-1,1]$.


\subsection{Some useful estimates}

We conclude this section with establishing several estimates that will essentially be used in the next
section.

\begin{lem}\label{lem:Jost.fpm}
There are constants $K_1$ and $K_2$ such that the following holds:
\begin{itemize}
  \item[(i)] for all sufficiently small $\eps$ and all $x$ with $|x|>x_\eps$,
  \[
        |\wt D^{-1}_\eps(k) \wt f_\pm(x,\eps,k)| \le K_1 |\rme^{\pm\rmi kx}|;
  \]
  \item[(ii)] for all sufficiently small $\eps$ and all $x\in\bR$,
  \begin{eqnarray}\nonumber
            \int_x^\infty \frac{|\wt f_+(t,\eps,k)|^2}{|\wt D_\eps(k)|^2}\,\rmd t
                &  \le K_2 |\rme^{2\rmi kx}|,\\
            \int_{-\infty}^x \frac{|\wt f_-(t,\eps,k)|^2}{|\wt D_\eps(k)|^2}\,\rmd t
                &\le K_2 |\rme^{-2\rmi kx}|.\nonumber
  \end{eqnarray}
\end{itemize}
\end{lem}

\begin{proof}
We only give details for the right Jost solution~$\wt f_+(\,\cdot\,,\eps,k)$, the case of $\wt
f_-(\,\cdot\,,\eps,k)$ being completely analogous. Item~(i) follows from the fact that
$\wt f(x,\eps,k)=\rme^{\rmi kx}$ for $x>x_\eps$ and that for $x<-x_\eps$
\[
    \frac{\wt f(x,\eps,k)}{\wt D_\eps(k)}
        = -\frac1{2ik}\, \rme^{\rmi kx} - \frac{\wt r_\eps(k)}{2ik}\, \rme^{-\rmi kx},
\]
where $\wt r_\eps(k)$ has a finite limit as $\eps\to0$ by Corollary~\ref{cor:prox.nres} in the
non-resonant case and by Corollary~\ref{cor:prox.res} in the resonant one.

To prove (ii), we first estimate the integral over $(-x_\eps,x_\eps)$. Denote by $P_\eps$ the operator of
multiplication by the function~$\chi_\eps$; then $P_\eps$ is an orthogonal projector and $\|P_\eps
h\|\to0$ as $\eps\to0$ for every $h\in L_2(\bR)$. In view of Theorem~\ref{thm:conv-tildeS}, we
find that
\[
    P_\eps(\wt S_\eps-k^2)^{-1} h = P_\eps\left[(\wt S_\eps-k^2)^{-1} - (S_0-k^2)^{-1}\right]h
                                    + P_\eps (S_0-k^2)^{-1} h \to 0
\]
in $L_2(\bR)$ as $\eps \to0$. Take now $h$ to be the characteristic function of the interval $[-3,-1]$;
then for all $\eps$ such that $x_\eps<1$ we calculate
\begin{eqnarray}\nonumber
    &P_\eps (\wt S_\eps-k^2)^{-1} h(x)
        = \dfrac{\chi_\eps(x)\wt f_+(x,\eps,k)}{\wt D_\eps(k)}
            \int_{-3}^{-1}\rme^{-\rmi kt}\,\rmd t \\\nonumber
       &\lefteqn{ = \dfrac{2\rme^{\rmi k}\sin k}k \,  \frac{\chi_\eps(x)\wt f_+(x,\eps,k)}{\wt D_\eps(k)}.}
\end{eqnarray}
It thus follows that
\[
    \|P_\eps (\wt S_\eps-k^2)^{-1} h\|^2 = \frac{4|\rme^{\rmi k}\sin k|^2}{|k|^2}
            \int_{-x_\eps}^{x_\eps} \frac{|f_+(x,\eps,k)|^2}{ |\wt D_\eps(k)|^2}\,\rmd x \to0
\]
as $\eps\to0$.

Returning now to part~(ii) of the lemma, we choose $\eps_0$ so small that $|\rme^{2\rmi kx}|>\tfrac12$
for all $x\le x_{\eps_0}$. If $x>x_{\eps_0}$, the desired inequality holds by (i) for all $\eps<\eps_0$ with $K_2 = K_1^2/(2|\myIm k|)$; otherwise we use (i) and the above limit to get the estimate
\begin{eqnarray}\nonumber
    &\displaystyle\int_x^\infty \frac{|\wt f_+(t,\eps,k)|^2}{|\wt D^2_\eps(k)|}\,\rmd t
            \le K_1^2 \int_x^\infty |\rme^{2\rmi kt}|\,\rmd t
                + \int_{-x_\eps}^{x_\eps}
                  \frac{|f_+(t,\eps,k)|^2}{ |\wt D_\eps(k)|^2}\,\rmd t \\\nonumber
            &\lefteqn{ \le K_2 |\rme^{2\rmi kx}|}
\end{eqnarray}
with $K_2:=1+K_1^2/(2|\myIm k|)$ holding for all sufficiently small~$\eps$.
\end{proof}

\section{Proximity of the operator families $S_\eps$ and $\wt S_\eps$}\label{sec:prox}

In this section we shall establish proximity of the operator families $S_\eps$ and $\wt S_\eps$ in the
norm resolvent sense, i.e., we shall prove the following theorem.

\begin{thm}\label{thm:res-proximity}
For every $k^2 \in \bC\setminus\bR$ it holds that
\begin{equation}\label{eq:prox.conv}
    \|(S_\eps - k^2)^{-1} - (\wt S_\eps - k^2)^{-1}\| \to 0
\end{equation}
as $\eps \to0$.
\end{thm}

Clearly, combination of Theorems~\ref{thm:conv-tildeS} and \ref{thm:res-proximity} gives a complete proof
of the claimed convergence of $S_\eps$ to~$S_0$.

Set
\[
    u_\eps(x):= \eps^{-1}|V(x/\eps)|^{1/2}\,[1 - \chi_\eps(x)]
\]
and
\[
    w_\eps(x):= {\rm sign}\,(V(x/\eps))\, u_\eps(x)
\]
and denote by $U_\eps$ and $W_\eps$ the operators of multiplications by $u_\eps$ and $w_\eps$
respectively. The operators $U_\eps$ and $W_\eps$ are in general unbounded; note, however, that the
$L_2$-norm of the functions $u_\eps$ and $w_\eps$ vanishes as $\eps\to0$. Indeed, in view of Lemma~\ref{lem:FMwithRho} we deduce
\begin{equation}\label{eq:norm-ueps}
  \begin{array}{l}
  \displaystyle   \|u_\eps\|^2 = \|w_\eps\|^2 = \frac1{\eps^2}\int_{|t|>x_\eps}
                                    \Bigl|V\Bigl(\frac{t}{\eps}\Bigr)\Bigr|\,\rmd t
                 = \frac1\eps\int_{|s|>\xi_\eps} |V(s)|\,\rmd s\\
  \displaystyle\qquad\qquad    \le \frac1{\eps}\int_{|s|>\xi_\eps}\frac{\rho_V(s)}{\rho_V(\xip)}|V(s)|\,\rmd s = \int_{|s|>\xi_\eps}\rho_V(s)|V(s)|\,\rmd s=o(1)
  \end{array}
\end{equation}
as $\eps\to0$.

The relation $S_\eps = \widetilde S_\eps + U_\eps W_\eps$ yields the following well-known (formal)
representation of the resolvent $R_\eps(k):=(S_\eps - k^2)^{-1}$ of $S_\eps$ in terms of the resolvent
$\widetilde R_\eps(k):=(\widetilde S_\eps - k^2)^{-1}$ of~$\widetilde S_\eps$ and the perturbation:
\begin{equation}\label{eq:prox.res}
    R_\eps(k) - \widetilde R_\eps(k) = \widetilde R_\eps(k) W_\eps
        \bigl[I + U_\eps \widetilde R_\eps(k) W_\eps \bigr]^{-1}
        U_\eps \widetilde R_\eps (k).
\end{equation}
We shall prove that the norms of the operators~$\widetilde R_\eps(k) W_\eps$, $U_\eps \wt R_\eps(k)
W_\eps$, and $U_\eps  \wt R_\eps (k)$ vanish as $\eps \to0$; this will justify both
formula~\eqref{eq:prox.res} and the claim of Theorem~\ref{thm:res-proximity}.


\subsection{The norm of special integral operators}
It is immediate to see that the operators~$\widetilde R_\eps(k) W_\eps$, $U_\eps \wt R_\eps(k) W_\eps$,
and $U_\eps  \wt R_\eps (k)$ are integral ones and act on a function $y\in L_2(\bR)$ via
\[
        \frac{\varphi_+(x)}{\wt D_\eps(k)} \int_{-\infty}^x \varphi_-(t) y(t)\,\rmd t
            + \frac{\varphi_-(x)}{\wt D_\eps(k)} \int_x^{\infty} \varphi_+(t) y(t)\,\rmd t,
\]
where $\varphi_+$ equals one of the functions $\wt f_+(\,\cdot\,,\eps,k)$, $\wt
f_+(\,\cdot\,,\eps,k)w_\eps$ or $\wt f_+(\,\cdot\,,\eps,k)u_\eps$, and similarly for $\varphi_-$. To
estimate the norm of such operators, we use the following result (see, e.g.,~\cite{Kufner-Persson:2003}):

\begin{prop}\label{pro:prox.T}
  Assume that $\psi_-$ and $\psi_+$ are functions belonging to $L_{2,\mathrm{loc}}^2(\bR)$ and let
  $T_\pm$ be integral operators defined via
  \[
        T_- y(x) = \psi_+(x) \int_{-\infty}^x \psi_-(t) y(t)\,\rmd t
  \]
  and
  \[
        T_+ y(x) = \psi_-(x) \int_x^{\infty} \psi_+(t) y(t)\,\rmd t
  \]
  respectively. Then $T_\pm$ are bounded in $L_2(\bR)$ if and only if the quantity
  \[
     K:= \sup_{x\in\bR} \biggl(\int_{-\infty}^x |\psi_-(t)|^2 \,\rmd t
                                \cdot \int_x^\infty |\psi_+(t)|^2\,\rmd t\biggr)^{1/2}
  \]
  is finite; in this case $\|T_\pm\| \le 2 K$.
\end{prop}


\subsection{The operator $U_\eps \wt R_\eps(k) W_\eps$} This is an integral operator of the form
\begin{eqnarray}\nonumber
    &&U_\eps \wt R_\eps(k) W_\eps y(x) = \dfrac{u_\eps(x)\wt f_+(x,\eps,k)}{\wt D_\eps(k)}
                                \int_{-\infty}^x y(t) w_\eps(t)\wt f_-(t,\eps,k)\,\rmd t \\\nonumber
                         &&\qquad\qquad\qquad\quad + \dfrac{u_\eps(x)\wt f_-(x,\eps,k)}{\wt D_\eps(k)}
                                \int_x^{\infty}  y(t) w_\eps(x)\wt f_+(t,\eps,k)\,\rmd t.
\end{eqnarray}
By virtue of Proposition~\ref{pro:prox.T} and the relation $|u_\eps|=|w_\eps|$ we conclude that the norm
of $U_\eps \wt R_\eps(k) W_\eps$ vanishes as $\eps\to0$ provided this is true for the quantity
$\sup_{x\in\bR} K_\eps(x)$, with
\[
    K_\eps(x) := \frac1{|\wt D_\eps(k)|^2}\int_{-\infty}^x |u_\eps(t)\wt f_-(t,\eps,k)|^2 \,\rmd t
                                \cdot \int_x^\infty |u_\eps(t)\wt f_+(t,\eps,k)|^2\,\rmd t.
\]

For $x\ge x_\eps$ we have
\[
    \int_x^\infty |u_\eps(t)\wt f_+(t,\eps,k)|^2\,\rmd t
        = \int_x^\infty |u_\eps(t)\rme^{\rmi kt}|^2\,\rmd t \le |\rme^{2\rmi kx}|\|u_\eps\|^2;
\]
similarly, by Lemma~\ref{lem:Jost.fpm}(i)
\begin{equation}\label{eq:uf-}
    \frac1{|\wt D_\eps(k)|^2}
                \int_{-\infty}^{x} |u_\eps(t)\wt f_-(t,\eps,k)|^2\,\rmd t
                \le K_1^2 \,|\rme^{-2\rmi kx}|\,\|u_\eps\|^2.
\end{equation}
Therefore,
\[
    \sup_{x\ge x_\eps} K_\eps(x) \le K_1^2 \|u_\eps\|^4 \to 0
\]
as $\eps\to0$ due to~\eqref{eq:norm-ueps}. For $x\le -x_\eps$ the arguments exploit the relation $\wt
f_-(x,\eps,k)=\rme^{-\rmi kx}$ and the estimate of Lemma~\ref{lem:Jost.fpm}\textit{(i)} to conclude that
\[
    \sup_{x\le -x_\eps} K_\eps(x) \le K_1^2 \|u_\eps\|^4 \to 0
\]
as $\eps\to0$. Finally,
\[
    \sup_{|x|<x_\eps} K_\eps(x) \le \frac{|\rme^{4\rmi kx_\eps}|\|u_\eps\|^4}
                    {|\wt D_\eps (k)|^2}\to0
\]
as $\eps\to0$, thus proving the following statement.

\begin{lem}\label{lem:URW-norm}
The norm of the operator $U_\eps \wt R_\eps(k) W_\eps$ vanishes as $\eps\to0$.
\end{lem}

\subsection{The operators $\wt R_\eps(k) W_\eps$ and $U_\eps \wt R_\eps(k)$} We need to show that both
$\sup_{x\in\bR}K_\eps^{(1)}(x)$ and $\sup_{x\in\bR}K_\eps^{(2)}(x)$ vanish as $\eps\to0$, with
\[
    K_\eps^{(1)}(x)
        := \frac1{\wt D^2_\eps(k)}\int_{-\infty}^x |u_\eps(t)\wt f_-(t,\eps,k)|^2 \,\rmd t
                                \cdot \int_x^\infty |\wt f_+(t,\eps,k)|^2\,\rmd t
\]
and
\[
    K_\eps^{(2)}(x)
        := \frac1{\wt D^2_\eps(k)}\int_{-\infty}^x |\wt f_-(t,\eps,k)|^2 \,\rmd t
                                \cdot \int_x^\infty |u_\eps(t)\wt f_+(t,\eps,k)|^2\,\rmd t.
\]
Since both quantities can be estimated in a similar way, only the first one will be treated in detail.

For $x\ge x_\eps$ we have
\[
    \int_x^\infty |\wt f_+(t,\eps,k)|^2\,\rmd t = \frac1{2|\myIm k|}\,|\rme^{2\rmi k x}|,
\]
which together with~\eqref{eq:uf-} results in
\[
    \sup_{x\ge x_\eps}K^{(1)}_\eps(x) \le \frac{K_1^2}{2|\myIm k|}\,\|u_\eps\|^2.
\]
For $x<x_\eps$ we use Lemma~\ref{lem:Jost.fpm}(ii) and the inequality
\[
    \int_{-\infty}^x |u_\eps(t) \wt f_-(t,\eps,k)|^2\,\rmd t
            \le |\rme^{-2\rmi kx}|\|u_\eps\|^2
\]
to conclude that
\[
    \sup_{x<x_\eps} K^{(1)}_\eps(x) \le  K_2 \|u_\eps\|^2.
\]

Combining the above estimates with~\eqref{eq:norm-ueps}, we arrive at the following conclusion.

\begin{lem}\label{lem:UR+RWnorm}
Under the standing assumptions,
\[
    \lim_{\eps\to0}\Bigl(\|U_\eps \wt R_\eps(k)\|+\|\wt R_\eps(k)W_\eps(k)\|\Bigr) =0.
\]
\end{lem}

\begin{proof}[Proof of Theorem~\ref{thm:res-proximity}] {It suffices to note that Lemmata~\ref{lem:URW-norm} and \ref{lem:UR+RWnorm} justify equation~\eqref{eq:prox.res} and, in turn, establish the convergence in~\eqref{eq:prox.conv}.}
\end{proof}

The main result of the paper is now easy to justify:

\begin{proof}[Proof of Theorem~\ref{thm:main}]
Clearly, the statements of Theorems~\ref{thm:conv-tildeS} and \ref{thm:res-proximity} immediately yield the norm resolvent convergence, as $\eps\to0$, of the family of Schr\"odinger operators~$S_\eps$ given by~\eqref{eq:intr.Seps} to the limiting operator $S_0$.
\end{proof}

\medskip

\emph{Acknowledgements.} The authors are grateful for to S.~Albeverio and C.~Cacciapuoti for bringing to their attention the papers~\cite{AlbeverioCacciapuotiFinco:2007, CacciapuotiExner:2007, Seba:1985} and stimulating discussions. They also thank the anonymous referee for careful reading of the manuscript and valuable remarks and suggestions. The second author acknowledges support from the Isaac Newton Institute for Mathematical Sciences at the University of Cambridge for participation in the programme \emph{``Inverse Problems''}, during which part of this work was done.

\end{document}